\documentclass{aptpub}

\authornames{Harry Crane} 
\shorttitle{$\varrho_{\nu}$-Markov partition process} 

\numberwithin{equation}{section}  

\def\per{\mathop{\rm per}\nolimits}

\def\PD{\mathop{\rm PD}\nolimits}

\def\masspartition{\mathop{\mathcal{P}_{\mathrm{m}}}\nolimits}
\def\masspartitionk{\mathop{\mathcal{P}_{\mathrm{m}}^{(k)}}\nolimits}
\def\partitionk{\mathop{\mathcal{P}^{(k)}}\nolimits}

\def\partitionsnk{\mathop{\mathcal{P}_{[n]}^{(k)}}\nolimits}

\def\subdot{\hbox{\bf.}}

\begin{document}

\title{A consistent Markov partition process generated from the paintbox process} 

\authorone[University of Chicago]{Harry Crane} 

\addressone{University of Chicago, Dept.\ of Statistics, Eckhart Hall Room 108, 5734 S.\ University Avenue,
Chicago, IL 60637 U.S.A.} 

\begin{abstract}
We study a family of Markov processes on $\mathcal{P}^{(k)}$, the space of partitions of the natural numbers with at most $k$ blocks.  The process can be constructed from a Poisson point process on $\mathbb{R}^+\times\prod_{i=1}^k\mathcal{P}^{(k)}$ with intensity $dt\otimes\varrho_{\nu}^{(k)}$, where $\varrho_{\nu}$ is the distribution of the paintbox based on the probability measure $\nu$ on $\masspartition$, the set of ranked-mass partitions of 1, and $\varrho_{\nu}^{(k)}$ is the product measure on $\prod_{i=1}^k\mathcal{P}^{(k)}$.  We show that these processes possess a unique stationary measure, and we discuss a particular set of reversible processes for which transition probabilities can be written down explicitly.
\end{abstract}


\section{Introduction}
Markov processes on the space of partitions appear in a variety of situations in scientific literature, such as, but not limited to, physical chemistry, astronomy, and population genetics.  See Aldous \cite{Aldous1999} for a relatively recent overview of this literature.  Well-behaved mathematically tractable models of random partitions are of interest to probabilists as well as statisticians and scientists, \cite{Ewens1972},\cite{Kingman1978},\cite{McCullagh2008},\cite{McCullagh2010}.  Ewens \cite{Ewens1972} first introduced the Ewens sampling formula in the context of theoretical population biology.  Kingman's \cite{Kingman1978} coalescent model was introduced as a model for population genetics, still its most natural setting.  However, since the seminal work of Ewens and Kingman, random partitions have appeared in areas ranging from classification models, as in \cite{Blei2003}, \cite{McCullagh2008}, to probability theory, see \cite{Bertoin2006},\cite{Pitman2005}.  McCullagh \cite{McCullagh2010} describes how the Ewens model can be used in the classical problem of estimating the number of unseen species, introduced by Fisher \cite{Fisher1943} and later studied by many, including Efron and Thisted \cite{EfronThisted1976}.

Berestycki \cite{Berestycki2004} studies a family of partition processes, called exchangeable fragmentation-coalescence (EFC) processes, whose paths are generated by a combination of independent coalescent and fragmentation processes.  The mathematical tractability of coalescent and fragmentation processes has led to the development of many results for EFC processes and has led to interest in more complex models.  For a sample of these results and relevant references see \cite{Bertoin2006},\cite{McCullaghPitman2008},\cite{Pitman2005}.  The study of processes, such as the EFC process, which admit a more general study of partition-valued processes is of interest from a theoretical as well as applied perspective.  In this paper, we study a family of processes which is similar in spirit to the EFC process, but whose sample paths are quite different.
\section{Preliminaries}\label{section:preliminaries}
Throughout this paper, $\mathcal{P}$ denotes the space of set partitions of the natural numbers $\mathbb{N}$.  We regard an element $B$ of $\mathcal{P}$ as a collection of disjoint non-empty subsets, called blocks, written $B=\{B_1,B_2,\ldots\}$, such that $\bigcup_{i}B_i=\mathbb{N}$.  The blocks are unordered, but, where necessary, they are listed in the order of their least element.  We write $B=(B_1,B_2,\ldots)$ whenever we wish to emphasize that blocks are listed in a particular order.  For $B\in\mathcal{P}$ and $b\in B$, $\#B$ is the number of blocks of $B$ and $\#b$ is the number of elements of $b$.  For any $A\subset\mathbb{N}$, let $B_{|A}$ denote the restriction of $B$ to $A$.  Wherever necessary, $\mathcal{P}^{(k)}$ denotes the space of partitions of $\mathbb{N}$ with at most $k$ blocks, i.e.\ $\mathcal{P}^{(k)}:=\{B\in\mathcal{P}:\#B\leq k\}$.  For fixed $n\in\mathbb{N}$, $\mathcal{P}_{[n]}$ and $\partitionsnk$ are the restriction to $[n]:=\{1,\ldots,n\}$ of $\mathcal{P}$ and $\mathcal{P}^{(k)}$ respectively.

%
%
It is sometimes convenient to regard a partition $B$ as either an equivalence relation defined by $B(i,j)=1\Leftrightarrow i\sim_{B} j$ or an $n\times n$ symmetric Boolean matrix whose $(i,j)$th entry is $B(i,j)$.  These three representations are equivalent and we use the same notation to refer to any one of them.

For each $\pi,\pi'\in\mathcal{P}$, we define the metric $d:\mathcal{P}\times\mathcal{P}\rightarrow\mathbb{R}$ such that
$$d(\pi,\pi')=1/\max\{n\in\mathbb{N}:\pi_{|[n]}=\pi'_{|[n]}\}.$$  The space $(\mathcal{P},d)$ is compact \cite{BertoinPIMS}.

In addition, we define the projection $D_{m,n}:\mathcal{P}_{[n]}\rightarrow\mathcal{P}_{[m]}$ for each $n\geq m\geq1$ by $D_{m,n}B_{[n]}=B_{[n]|[m]}$.  In the matrix representation, $D_{m,n}B$ is the leading $m\times m$ sub-matrix of $B$.  We seek processes $B:=(B_t,t\geq0)$ on $\mathcal{P}$ such that for each $n\in\mathbb{N}$, the restriction of $B$ to $[n]$, $B_{|[n]},$ is finitely exchangeable and consistent.  That is,
\begin{itemize}
	\item $\sigma(B_{|[n]})\sim B_{|[n]}$ for each $\sigma\in\mathcal{S}_{n}$, the symmetric group acting on $[n]$, and
	\item $B_{[n]|[m]}\sim B_{|[m]}$ for each $m<n$.
\end{itemize}
It is more convenient to work with $\mathcal{P}$ as the state space of our process than the space $\masspartition=\{(s_1,s_2,\ldots):s_1\geq s_2\geq\ldots\geq0,\mbox{ }\sum_{i}s_i\leq1\}$ of ranked-mass partitions of $x\in[0,1]$.  In accordance with the notation for set partitions, let $\masspartitionk:=\{s\in\masspartition:s_j=0\mbox{ }\forall j>k,\mbox{ }\sum_{i=1}^ks_i=1\}$ denote the ranked $k$-simplex.  There is an intimate relationship between exchangeable processes on $\mathcal{P}$ and processes on $\masspartition$ through the paintbox process.

For $s\in\masspartition$, let $X:=(X_1,X_2,\ldots)$ be independent random variables with distribution 
\begin{displaymath}\mathbb{P}_s(X_i=j)=\left\{\begin{array}{cc}
s_j, & j\geq1\\
1-\sum_{i=1}^{\infty}s_i, & j=-i\\
0,&\mbox{o.w.}\end{array}\right.\end{displaymath}  The partition $\Pi(X)$ generated by $s$ through $X$ satisfies $i\sim_{\Pi(X)}j$ if and only if $X_i=X_j.$  The distribution of $\Pi(X)$ is written $\varrho_s$ and $\Pi(X)$ is called the paintbox based on $s$.  For a measure $\nu$ on $\masspartition$, the paintbox based on $\nu$ is the $\nu$-mixture of paintboxes, written $\varrho_{\nu}(\cdot):=\int_{\masspartition}\varrho_{s}(\cdot)\nu(ds).$  Any partition obtained in this way is an exchangeable random partition of $\mathbb{N}$ and every infinitely exchangeable partition admits a representation as the paintbox generated by some $\nu$. See \cite{BertoinPIMS} and \cite{Pitman2005} for more details on the paintbox process.

We are particularly interested in exchangeable Markovian transition probabilities $(p_n)$, where, for every $n$, $p_n$ is a transition probability on $\mathcal{P}_{[n]}$ which satisfies
\begin{equation}p_n(B,B')=\sum_{B''\in D^{-1}_{n,n+1}(B')}p_{n+1}(B^*,B''),\label{eq:consistent tps}\end{equation}
for each $B,B'\in\mathcal{P}_{[n]}$ and $B^*\in D^{-1}_{n,n+1}(B)$.  Burke and Rosenblatt \cite{BurkeRosenblatt1958} show that \eqref{eq:consistent tps} is necessary and sufficient for $(p_n)$ to be consistent under selection from $\mathbb{N}$.

Likewise, for a continuous-time Markov process, $(B_n(t),t\geq0)_{n\in\mathbb{N}}$, where $B_n(t)$ is a process on $\mathcal{P}_{[n]}$ with infinitesimal generator $Q_n$, it is sufficient that the entries of $Q_n$ satisfy \eqref{eq:consistent tps} for there to be a Markov process on $\mathcal{P}$ with those finite-dimensional transition rates.

\section{The $\varrho_{\nu}$-Markov chain on $\mathcal{P}^{(k)}$}\label{section:markov chain}
Let $n,k\in\mathbb{N}$ and let $\nu$ be a probability measure on the ranked $k$-simplex $\masspartitionk$, so that the paintbox based on $\nu$ is obtained by a conditionally i.i.d.\ sample from $\nu$, i.e.\ given $s\sim\nu$, $X_1,X_2,\ldots$ are i.i.d.\ with $\mathbb{P}_{s}(X_i=j)=s_j$ for each $j=1,\ldots,k$.  For convenience, we write $B\in\mathcal{P}^{(k)}$ as an ordered list $(B_1,\ldots,B_k)$ where $B_i$ corresponds to the $i$th block of $B$ in order of appearance for $i\leq\#B$ and $B_i=\emptyset$ for $i=\#B+1,\ldots,k$. 

Consider the following Markov transition operation $B\mapsto B'$ on $\mathcal{P}^{(k)}$.  Let $B=(B_1,\ldots,B_k)\in\mathcal{P}^{(k)}$ and, independently of $B$, generate $C_1,C_2,\ldots$ which are independent and identically distributed accord to $\varrho_{\nu}$.  For each $i$, we write $C_i:=(C_{i1},\ldots,C_{ik})\in\mathcal{P}^{(k)}$.  Independently of $B,C_1,C_2,\ldots,$ generate $\sigma_1,\sigma_2,\ldots$ which are independent uniform random permutations of $[k]$.  Given $\sigma:=(\sigma_1,\sigma_2,\ldots,\sigma_k)$, we arrange $B,C_1,\ldots,C_k$ in matrix form as follows:
\begin{displaymath}
\bordermatrix{\text{}&C_{\subdot 1} & C_{\subdot 2} & \ldots & C_{\subdot k}\cr
B_1 & C_{1,\sigma_1(1)}\cap B_1 & C_{1,\sigma_1(2)}\cap B_1&\ldots& C_{1,\sigma_1(k)}\cap B_1\cr
B_2 & C_{2,\sigma_2(1)}\cap B_2 & C_{2,\sigma_2(2)}\cap B_2&\ldots& C_{2,\sigma_2(k)}\cap B_2\cr
\vdots & \vdots & \vdots & \ddots & \vdots\cr
B_k & C_{k,\sigma_k(1)}\cap B_k & C_{k,\sigma_k(2)}\cap B_k & \ldots & C_{k,\sigma_k(k)}\cap B_k}=:B\cap C^{\sigma}.
\end{displaymath}
$B\cap C^{\sigma}$ is a matrix with row totals corresponding to the blocks of $B$ and column totals $C_{\subdot j}=\bigcup_{i=1}^k(C_{i,\sigma_i(j)}\cap B_i)$.  Finally, $B'$ is obtained as the collection of non-empty blocks of $(C_{\subdot1},\ldots,C_{\subdot k})$.  The non-empty entries of $B\cap C^{\sigma}$ form a partition in $\mathcal{P}^{(k^2)}$ which corresponds to the greatest lower bound $B\wedge B'$.
\begin{prop}\label{prop:paintbox tps} The above description gives rise to finite-dimensional transition probabilities on $\partitionsnk$
\begin{equation}p_n(B,B';\nu)=\frac{k!}{(k-\#B')!}\prod_{b\in B}\frac{(k-\#B'_{|b})!}{k!}\varrho_{\nu}(B'_{|b}).\label{eq:paintbox tps}\end{equation}
\end{prop}
\begin{proof}
Let $A\in\mathcal{P}^{(k)}$.  Fix $n,k\in\mathbb{N}$, put $B:=A_{|[n]}\in\partitionsnk$.  Let $C_1,\ldots,C_k$ be i.i.d.\ $\varrho_{\nu}$-distributed partitions and $\sigma:=(\sigma_1,\ldots,\sigma_k)$ i.i.d.\ uniform random permutations of $[k]$ as described above.  Let $B'$ be the set partition obtained from the column totals of the matrix $B\cap C^{\sigma}$ in the above construction.

From the matrix construction, it is clear that for each $i=1,\ldots,k$, the restriction $B'_{|B_i}$ is equal to the set partition in $\partitionsnk$ associated with $C_i[B_i]:=(C_{i1}\cap B_i,\ldots, C_{ik}\cap B_i)$.  Conversely, the transition $B\mapsto B'$ occurs only if the collection $(C_1,\ldots,C_k)$ is such that, for each $B_i\in B$, $C_i[B_i]=B'_{|B_i}$.  By consistency of the paintbox process, for each $i=1,\ldots,k$, $C_i[B_{i}]$ has probability $$\varrho_{\nu}(C_i[B_{i}])=\varrho_{\nu}(B'_{|B_i}).$$  Independence of the $C_i$ implies that the probability of $B\wedge B'$ given $B$ is $$\prod_{b\in B}\varrho_{\nu}(B'_{|b}).$$  Finally, each uniform permutation $\sigma_i$ has probability $1/k!$ and there are $\frac{k!}{(k-\#B')!}\prod_{b\in B}(k-\#B'_{|b})!$ collections $\sigma_1,\ldots,\sigma_{\#B}$ such that the column totals of $B\cap C^{\sigma}$ correspond to the blocks of $B'$.  This completes the proof.
\end{proof}
For fixed $n$, \eqref{eq:paintbox tps} only depends on $B$ and $B'$ through $\varrho_{\nu}$ and the number of blocks of $B$ and $B'$ and is, therefore, finitely exchangeable.  I appeal to \eqref{eq:consistent tps} to establish consistency.
\begin{prop}\label{prop:consistent tps}For any measure $\nu$ on $\masspartitionk$, let $(p_n(\cdot,\cdot;\nu))_{n\geq1}$ be the collection of transition probabilities on $\partitionsnk$ defined in \eqref{eq:paintbox tps}.  Then $(p_n)$ is a consistent family of transition probabilities.\end{prop}
\begin{proof}
Fix $n,k\in\mathbb{N}$ and let $B,B'\in\partitionsnk$.  To establish consistency it is enough to verify condition \eqref{eq:consistent tps} from theorem 1 of \cite{BurkeRosenblatt1958}, i.e.\ for each $\nu$ and $B^*\in D^{-1}_{n,n+1}(B)$,
$$p_{n+1}(B^*,D^{-1}_{n,n+1}(B');\nu)=p_n(B,B';\nu).$$  
We assume without loss of generality that $B^*\in D^{-1}(B)$ is obtained from $B$ by the operation $n+1\mapsto B_1\in B$ and we write $B^*_1:=B_1\cup\{n+1\}$.  Likewise, for $B''\in D^{-1}_{n,n+1}(B')$ obtained by $n+1\mapsto B'_i\in B'\cup\{\emptyset\}$, write $B'^*_i:=B'_i\cup\{n+1\}$.  So either $n+1\in B'^*_i$ for some $i=1,\ldots,\#B'$ or $n+1$ is inserted in $B'$ as a singleton.

The change to $B\cap C^{\sigma}$ that results from inserting $n+1$ into $B_1\in B$ and $B'_i\in B'$ is summarized by the following matrix.  Note that $B'_j=\emptyset$ for $j>\#B'$.
\begin{displaymath}
\bordermatrix{\text{}&B'_{1} & B'_{2} & \ldots &B'^*_{i} &\ldots & B'_{k}\cr
B^*_1 & B'_1\cap B_1 & B'_2\cap B_1&\ldots& (B'_i\cap B_1)\cup\{n+1\}&\ldots & B'_k\cap B_1\cr
B_2 & B'_1\cap B_2 & B'_2\cap B_2&\ldots& B'_i\cap B_2 &\ldots &B'_k\cap B_2\cr
\vdots & \vdots & \vdots & \ddots &\vdots & \ddots & \vdots\cr
B_k & B'_1\cap B_k & B'_2\cap B_k & \ldots & B'_i\cap B_k&\ldots & B'_k\cap B_k}.
\end{displaymath}
Here, the blocks of $B$ are listed in any order, with empty sets inserted as needed, and the blocks of $B'$ are listed in order of least elements, with $k-\#B'$ empty sets at the end.

Given $B'$, the set of compatible partitions $D^{-1}_{n,n+1}(B')$ consists of three types depending on the subset $B_1\subset[n]$ and the block of $B'$ into which $\{n+1\}$ is inserted.  Let $B''\in D^{-1}_{n,n+1}(B')$ be the partition of $[n+1]$ obtained by inserting $n+1$ in $B'$.  Either
\begin{enumerate}
	\item[(i)] $n+1$ is inserted into a block $B'_i$ such that $B'_i\cap B_1\neq\emptyset\Rightarrow\#B''_{|B_1^*}=\#B'_{|B_1}$,
	\item[(ii)] $n+1$ is inserted into a block $B'_i\neq\emptyset$ such that $B_i'\cap B_1=\emptyset\Rightarrow\#B''_{|B_1^*}=\#B'_{|B_1}+1$, or
	\item[(iii)] $n+1$ is inserted into $B'$ as a singleton block $\Rightarrow\#B''_{|B_1^*}=\#B'_{|B_1}+1$ and $\#B''=\#B'+1$; we denote this partition by $B'_{\emptyset}$.
\end{enumerate}
There are $k-\#B'$ empty columns in which $\{n+1\}$ can be inserted as a singleton in $B'$, as in (iii).  For $B''$ obtained by (ii), the restriction of $B''$ to $B_1^*$ coincides with the restriction of $B'_{\emptyset}$ to $B_1^*$, so each of these restrictions has the same probability under $\varrho_{\nu}$.  For notational convenience in the following calculation, let $\mathcal{D}_1$ be those elements of $D^{-1}_{n,n+1}(B')$ which satisfy condition (i) above and $\mathcal{D}_2$ those which satisfy condition (ii).\\\small
\begin{eqnarray}
\lefteqn{p_{n+1}(B^*,D^{-1}_{n,n+1}(B');\nu)=}\notag\\
&&\sum_{B''\in D^{-1}_{n,n+1}(B')}\frac{k!}{(k-\#B'')!}\prod_{b\in B^*}\frac{(k-\#B''_{|b})!}{k!}\varrho_{\nu}(B''_{|b})\label{line 1}\\
&=&\frac{k!}{(k-\#B')!}\prod_{b\in B}\frac{(k-\#B'_{|b})!}{k!}\left[\sum_{B''\in\mathcal{D}_1}\prod_{b\in B^*}\varrho_{\nu}(B''_{|b})+\right.\notag\\
&&\quad\left. +\sum_{B''\in\mathcal{D}_2}\frac{1}{k-\#B'_{|B_1}}\prod_{b\in B^*}\varrho_{\nu}(B''_{|b})+\frac{k-\#B'}{k-\#B'_{|B_1}}\prod_{b\in B^*}\varrho_{\nu}(B'_{\emptyset|b})\right]\label{line 2}\\
&=&\frac{k!}{(k-\#B')!}\prod_{b\in B}\frac{(k-\#B'_{|b})!}{k!}\prod_{b\in B^*:b\neq B^*_1}\varrho_{\nu}(B'_{|b})\left[\sum_{B''\in\mathcal{D}_1}\varrho_{\nu}(B''_{|B_1^*})+\right.\notag\\
&&\quad\left.+\sum_{B''\in\mathcal{D}_2}\frac{1}{k-\#B'_{|B_1}}\varrho_{\nu}(B''_{|B_1^*})+\frac{k-\#B'}{k-\#B'_{|B_1}}\varrho_{\nu}(B'_{\emptyset|B_1^*})\right]\notag\\
&=&\frac{k!}{(k-\#B')!}\prod_{b\in B}\frac{(k-\#B'_{|b})!}{k!}\prod_{b\in B:b\neq B^*_1}\varrho_{\nu}(B'_{|b})\left[\sum_{B''\in\mathcal{D}_1}\varrho_{\nu}(B''_{|B_1^*})+\varrho_{\nu}(B'_{\emptyset|B_1^*})\right]\label{line 4}\\
&=&\frac{k!}{(k-\#B')!}\prod_{b\in B}\frac{(k-\#B'_{|b})!}{k!}\prod_{b\in B:b\neq B^*_1}\varrho_{\nu}(B'_{|b})\left[\sum_{B''\in D^{-1}_{\#B_1,\#B_1+1}(B'_{|B_1})}\varrho_{\nu}(B'')\right]\label{line 5}\\
&=&\frac{k!}{(k-\#B')!}\prod_{b\in B}\frac{(k-\#B'_{|b})!}{k!}\prod_{b\in B:b\neq B^*_1}\varrho_{\nu}(B'_{|b})\left[\varrho_{\nu}(B'_{|B_1})\right]\label{line 6}\\
&=&\frac{k!}{(k-\#B')!}\prod_{b\in B}\frac{(k-\#B'_{|b})!}{k!}\prod_{b\in B}\varrho_{\nu}(B'_{|b})\notag\\
&=&p_n(B,B';\nu).\notag
\end{eqnarray}\normalsize
Here, \eqref{line 2} is obtained from \eqref{line 1} by factoring $\frac{k!}{(k-\#B')!}\prod_{b\in B}\frac{(k-\#B'_{|b})!}{k!}$ out of the sum and using observations (i), (ii) and (iii).  In \eqref{line 4}, we use the fact that for any $B''\in\mathcal{D}_2$, $B''_{|B^*_1}=B'_{\emptyset|B_1^*}$, and there are $\#B'-\#B'_{|B_1}$ elements in $\mathcal{D}_2$ according to (ii).  Line \eqref{line 5} follows by observing that each $B''\in\mathcal{D}_1$ corresponds to an element of $D^{-1}_{\#B_1,\#B_1+1}(B'_{|B_1})$ and $B'_{\emptyset|B_1^*}$ is the element of $D^{-1}_{\#B_1,\#B_1+1}(B'_{|B_1})$ obtained by inserting $\{n+1\}$ as a singleton in $B'_{|B_1}$.  Finally, \eqref{line 6} follows from \eqref{line 5} by consistency of the paintbox process.  This completes the proof.
\end{proof}
The following result is immediate by finite exchangeability and consistency of \eqref{eq:paintbox tps} for every $n$ and Kolmogorov's extension theorem (theorem 36.1, \cite{Billingsley1995}).
\begin{thm}\label{thm:existence chain}There exists a transition probability $p(\cdot,\cdot;\nu)$ on $\left(\mathcal{P}^{(k)},\sigma\left(\bigcup_n\partitionsnk\right)\right)$ whose finite-dimensional restrictions are given by \eqref{eq:paintbox tps}.\end{thm}
We call the discrete-time process governed by $p(\cdot,\cdot;\nu)$ the $\varrho_{\nu}$-Markov chain with state space $\mathcal{P}^{(k)}$.
\subsection{Equilibrium measure}\label{section:equilibrium measure}
From \eqref{eq:paintbox tps}, it is clear that for each $n,k\in\mathbb{N}$ and $B,B'\in\partitionsnk$, $p_n(B,B';\nu)$ is strictly positive provided $\nu$ is such that $\nu(s)>0$ for some $s=(s_1,\ldots,s_k)\in\masspartitionk$ with $s_k>0$.  Under this condition, the finite-dimensional chains are aperiodic and irreducible on $\partitionsnk$ and, therefore, have a unique stationary distribution.  In fact, the finite-dimensional chains based on $\nu$ are aperiodic and irreducible provided $\nu$ is not degenerate at $(1,0,\ldots,0)\in\masspartitionk$.  The existence of a unique stationary distribution for each $n$ implies that there is a unique stationary probability measure on $\left(\mathcal{P}^{(k)},\sigma\left(\bigcup_n\partitionsnk\right)\right)$ for $p(\cdot,\cdot;\nu)$ from theorem \ref{thm:existence chain}.
\begin{prop}\label{prop:existence stationary dist} Let $\nu$ be a measure on $\masspartitionk$ such that $\nu$ is non-degenerate at $(1,0,\ldots,0)\in\masspartitionk$.  Then there exists a unique stationary distribution $\theta_n(\cdot;\nu)$ for $p_n(\cdot,\cdot;\nu)$ for each $n\geq1$.\end{prop}
\begin{proof}
Fix $n\in\mathbb{N}$ and let $\nu$ be any measure on $\masspartitionk$ other than that which puts unit mass at $(1,0,\ldots,0)$.  For $B=(B_1,\ldots,B_m)\in\partitionsnk$, \eqref{eq:paintbox tps} gives the transition probability
$$p_n(B,B;\nu)=\frac{k!}{(k-m)!}\prod_{i=1}^m\frac{1}{k}\varrho_{\nu}(B_i)$$
and $\varrho_{\nu}(B_i)=\varrho_{\nu}([\#B_i])>0$ for each $i=1,\ldots,m.$  Hence, $p_n(B,B;\nu)>0$ for every $B\in\partitionsnk$ and the chain is aperiodic.

To see that the chain is irreducible, let $B,B'\in\partitionsnk$ and let $1_n$ denote the one block partition of $[n]$.  Then 
$$p_n(B,1_n;\nu)=k\prod_{b\in B}\frac{1}{k}\varrho_{\nu}([\#b])>0$$
and, since $\nu$ is not degenerate at $(1,0,\ldots,0)$, there exists a path $1_n\mapsto B'$ by recursively partitioning $1_n$ until it coincides with $B'$.  For instance, let $B':=(B'_1,\ldots,B'_m)\in\mathcal{P}^{(k)}$.  One such path from $1_n$ to $B'$ is
$$1_n\rightarrow(B_1',\bigcup_{i=2}^m B'_i)\rightarrow (B_1',B_2',\bigcup_{i=3}^m)\rightarrow\cdots\rightarrow B'$$
which has positive probability for any non-degenerate $\nu$.
  Hence $p_n(\cdot,\cdot;\nu)$ is irreducible, which establishes the existence of a unique stationary distribution for each $n$.
\end{proof}
\begin{thm}\label{thm:stationary distribution}Let $\nu$ be a measure on $\masspartitionk$ such that $\nu((1,0,\ldots,0))<1$.  Then there exists a unique stationary probability measure $\theta(\cdot;\nu)$ for the $\varrho_{\nu}$-Markov chain on $\mathcal{P}^{(k)}$.\end{thm}
\begin{proof} For $\nu$ satisfying this condition, proposition \ref{prop:existence stationary dist} shows that a stationary distribution exists for each $n\geq1$.  Let $(\theta_n(\cdot;\nu),n\geq1)$ be the collection of stationary distributions for the finite-dimensional transition probabilities $(p_n(\cdot,\cdot;\nu),n\geq1)$.  We now show that the $\theta_n$ are consistent and finitely exchangeable for each $n$.

Fix $n\in\mathbb{N}$ and let $B\in\partitionsnk$.  Then stationarity of $\theta_n(\cdot;\nu)$ implies
$$\sum_{B'\in\partitionsnk}\theta_n(B';\nu)p_n(B',B;\nu)=\theta_n(B;\nu).$$  
Now write $\theta_n(\cdot)\equiv\theta_n(\cdot;\nu)$ and $p_n(\cdot,\cdot)\equiv p_n(\cdot,\cdot;\nu)$ for convenience and let $B'\in\partitionsnk$.  We have\\
\begin{eqnarray*}
\underbrace{\sum_{B''\in D^{-1}_{n,n+1}(B')}\theta_{n+1}(B'')}_{(\theta_{n+1}D^{-1}_{n,n+1})(B')}&=&\sum_{B''\in D^{-1}_{n,n+1}(B')}\sum_{B^*\in\mathcal{P}_{[n+1]}^{(k)}}\theta_{n+1}(B^*)p_{n+1}(B^*,B'')\\
&=&\sum_{B^*\in\mathcal{P}_{[n+1]}^{(k)}}\theta_{n+1}(B^*)\left[\sum_{B''\in D^{-1}_{n,n+1}(B')}p_{n+1}(B^*,B'')\right]\\
&=&\sum_{B\in\partitionsnk}\sum_{B^*\in D^{-1}_{n,n+1}(B)}\theta_{n+1}(B^*)\left[p_n(B,B')\right]\\
&=&\sum_{B\in\partitionsnk}p_n(B,B')\sum_{B^*\in D^{-1}_{n,n+1}(B)}\theta_{n+1}(B^*)\\
&=&\sum_{B\in\partitionsnk}p_n(B,B')(\theta_{n+1}D^{-1}_{n,n+1})(B).\end{eqnarray*}
So we have that $\theta_{n+1}D^{-1}_{n,n+1}$ is stationary for $p_n$ which implies that $\theta_n\equiv\theta_{n+1}D^{-1}_{n,n+1}$ by uniqueness and $\theta_n$ is consistent for each $n$.  

Let $\sigma$ be a permutation of $[n]$.  Then for any $B,B'\in\partitionsnk$, $p_n(\sigma(B),\sigma(B'))=p_n(B,B')$ by exchangeability of $p_n$.  It follows that $\theta_n$ is finitely exchangeable for each $n$ since
$$\sum_{B\in\partitionsnk}\theta_n(\sigma(B))p_n(\sigma(B),\sigma(B'))=\theta_n(\sigma(B'))$$ by stationarity, and $p_n(\sigma(B),\sigma(B'))=p_n(B,B')$ implies that 
$$\sum_{B\in\partitionsnk}\theta_n(\sigma(B))p_n(B,B')=\theta_n(\sigma(B')).$$  Hence, $\theta_n\circ\sigma$ is stationary for $p_n$ and $\theta_n\equiv\theta_n\circ\sigma$ by uniqueness.

Kolmogorov consistency implies that there exists a unique exchangeable stationary probability measure $\theta$ on $\mathcal{P}^{(k)}$ whose restriction to $[n]$ is $\theta_n$ for each $n\in\mathbb{N}$.  This completes the proof.
\end{proof}
\section{The $\varrho_{\nu}$-Markov process in continuous time}\label{section:continuous-time process}
Let $\lambda>0$, $\nu$ be a measure on $\masspartitionk$ and for each $n\in\mathbb{N}$ define Markovian infinitesimal jump rates for a Markov process on $\partitionsnk$ by
\begin{eqnarray}q_n(B,B';\nu)=\left\{\begin{array}{cc}
\lambda p_n(B,B';\nu), & B\neq B'\\
0, & \mbox{o.w.}\end{array}\right.\label{eq:rates}\end{eqnarray}
where $p_n$ is as in \eqref{eq:paintbox tps}.  The infinitesimal generator, $Q^{\nu}_n$, of the process on $\partitionsnk$ governed by $q_n$ has entries
\begin{eqnarray}Q^{\nu}_n(B,B')=\lambda\times\left\{\begin{array}{cc}
p_n(B,B';\nu), & B\neq B'\\
p_n(B,B;\nu)-1, & B=B'.\end{array}\right.\label{eq:infinitesimal generator}\end{eqnarray}
We now construct a Markov process $B:=(B(t),t\geq0)$ in continuous time whose finite-dimensional transition rates are given by \eqref{eq:rates}.
\begin{defn}\label{defn:process}A process $B:=(B(t),t\geq0)$ on $\mathcal{P}^{(k)}$ is a $\varrho_{\nu}$-Markov process if, for each $n\in\mathbb{N}$, $B_{|[n]}$ is a Markov process on $\partitionsnk$ with $Q$-matrix $Q^{\nu}_n$ as in \eqref{eq:infinitesimal generator}.
\end{defn}
A process on $\mathcal{P}^{(k)}$ whose finite-dimensional restrictions are governed by $Q^{\nu}_n$ can be constructed according to the matrix construction from section 3 by permitting only transitions $B\mapsto B'$ for $B'\neq B$, where $B,B'\in\partitionsnk$, and adding a hold time which is exponentially distributed with mean $-1/Q^{\nu}_n(B,B).$
\begin{prop}\label{prop:consistent Q-matrix}For a measure $\nu$ on $\masspartitionk$, let $(Q^{\nu}_n)_{n\in\mathbb{N}}$ be the collection of $Q$-matrices in \eqref{eq:infinitesimal generator}.  For every $n\in\mathbb{N}$, the entries of $Q^{\nu}_n$ satisfy \eqref{eq:consistent tps}.\end{prop}
\begin{proof}
Fix $n\in\mathbb{N}$ and let $B,B'\in\partitionsnk$ such that $B\neq B'$.  Then 
$$Q^{\nu}_n(B,B')=\sum_{B''\in D^{-1}_{n,n+1}(B')}Q^{\nu}_{n+1}(B_*,B'')$$ for all $B_*\in D^{-1}_{n,n+1}(B)$ by the consistency of $p_n$ from proposition \ref{prop:consistent tps}.

For $B'=B$ and $B_*\in D^{-1}_{n,n+1}(B)$, we have\small
\begin{eqnarray*}
\lefteqn{\sum_{B''\in D^{-1}_{n,n+1}(B)}Q^{\nu}_{n+1}(B_*,B'')=}\\
&&Q^{\nu}_{n+1}(B_*,B_*)+\sum_{B''\in D^{-1}_{n,n+1}(B)\backslash\{B_*\}}Q^{\nu}_{n+1}(B_*,B'')\\
&=&\lambda\left[p_{n+1}(B_*,B_*;\nu)-1+\sum_{B''\in D^{-1}_{n,n+1}(B)\backslash\{B_*\}}p_{n+1}(B_*,B'';\nu)\right]\\
&=&\lambda\left[\sum_{B''\in D^{-1}_{n,n+1}(B)}p_{n+1}(B_*,B'';\nu)-1\right]\\
&=&\lambda(p_n(B,B;\nu)-1)\\
&=&Q^{\nu}_n(B,B).\end{eqnarray*}\normalsize
\end{proof}
\begin{thm}\label{thm:existence continuous-time}For each measure $\nu$ on $\masspartitionk$, there exists a Markov process $(B(t),t\geq0)$ on $\mathcal{P}^{(k)}$ which has finite-dimensional transition rates given in \eqref{eq:rates}.\end{thm}
\begin{proof}
Let $\nu$ be a measure on $\masspartitionk$ and $(B_{|[n]}(t),t\geq0)_{n\in\mathbb{N}}$ be the collection of restrictions of a $\varrho_{\nu}$-Markov process with consistent $Q$-matrices $(Q^{\nu}_n)_{n\in\mathbb{N}}$ as in \eqref{eq:infinitesimal generator}.  For each $n$, $Q^{\nu}_n$ is finitely exchangeable and consistent with $Q^{\nu}_{n+1}$ by proposition \ref{prop:consistent Q-matrix}, which is sufficient for $B_{|[n]}$ to be consistent with $B_{|[n+1]}$ for every $n$.  Kolmogorov's extension theorem implies that there exist transition rates, $Q^{\nu}$, on $\mathcal{P}^{(k)}$ such that for every $B,B'\in\partitionsnk$,
$$Q^{\nu}_n(B,B')=Q^{\nu}(B_*,\{B''\in\mathcal{P}^{(k)}:B''_{|[n]}=B'\}),$$ for every $B_*\in\{B''\in\mathcal{P}^{(k)}:B''_{|[n]}=B\}$.

   Finally, for every $B\in\partitionsnk$, $Q^{\nu}_n(B,\partitionsnk\backslash\{B\})=\lambda(1-p_n(B,B;\nu))<\infty$ so that the sample paths of $B_{|[n]}$ are c\`adl\`ag for every $n$, which implies that $B$ is c\`adl\`ag.
\end{proof}
\begin{cor}\label{cor:continuous-time stationary}For $\nu$ which satisfies the condition of theorem \ref{thm:stationary distribution}, the continuous-time process $B:=(B(t),t\geq0)$ with finite-dimensional rates $q_n(\cdot,\cdot;\nu)$ in \eqref{eq:rates} has unique stationary distribution $\theta(\cdot;\nu)$ from theorem \ref{thm:stationary distribution}.\end{cor}
\begin{proof} For each $n\in\mathbb{N}$, let $\theta_n(\cdot;\nu)$ be the unique finite-dimensional stationary distribution of $p_n(\cdot,\cdot;\nu)$ from \eqref{eq:paintbox tps}.
It is easy to verify that for each $n\in\mathbb{N}$, $\Theta^{\nu}_n:=(\theta_n(B;\nu),B\in\partitionsnk)$ satisfies
$$\left(\Theta^{\nu}_n\right)^tQ^{\nu}_n=0,$$ which establishes that $\Theta^{\nu}_n$ is stationary for $Q^{\nu}_n$ for every $n$.  The rest follows by theorem \ref{thm:stationary distribution}.\end{proof}
\subsection{Poissonian construction}\label{section:poissonian construction}
From the matrix construction at the beginning of section \ref{section:markov chain}, a consistent family of finite-dimensional Markov processes with transition rates as in \eqref{eq:rates} can be constructed by a Poisson point process on $\mathbb{R}^+\times\prod_{i=1}^k\partitionk$ as follows.  Let $P=\{(t,C_1,\ldots,C_k)\}\subset\mathbb{R}^+\times\prod_{i=1}^k\partitionk$ be a Poisson point process with intensity measure $dt\otimes\lambda\varrho_{\nu}^{(k)}$ for some measure $\nu$ on $\masspartitionk$ and $\lambda>0$, where $\varrho_{\nu}^{(k)}$ is the product measure $\varrho_{\nu}\otimes\cdots\otimes\varrho_{\nu}$ on $\prod_{i=1}^k\partitionk$.

Construct an exchangeable process $B:=(B(t),t\geq0)$ on $\mathcal{P}^{(k)}$ by taking $\pi\in\mathcal{P}^{(k)}$ to be some exchangeable random partition and setting $B(0)=\pi$.

For each $n\in\mathbb{N}$, put $B_{|[n]}(0)=\pi_{|[n]}$ and
\begin{itemize}
	\item if $t$ is not an atom time for $P$, then $B_{|[n]}(t)=B_{|[n]}(t-)$;
	\item if $t$ is an atom time for $P$ so that $(t,C_1,\ldots,C_k)\in P$, then, independently of $(B(s),s<t)$ and $(t,C_1,\ldots,C_k)$ generate $\sigma_1,\ldots,\sigma_k$ i.i.d.\ uniform random permutations of $[k]$ and construct $B'$ from the set partition induced by the column totals $(C_{\subdot1},\ldots,C_{\subdot k})$ of 
\begin{displaymath}
\bordermatrix{\text{}&C_{\subdot1} & C_{\subdot2} & \ldots & C_{\subdot k}\cr
B_1 & C_{1,\sigma_1(1)}\cap B_1 & C_{1,\sigma_1(2)}\cap B_1&\ldots& C_{1,\sigma_1(k)}\cap B_1\cr
B_2 & C_{2,\sigma_2(1)}\cap B_2 & C_{2,\sigma_2(2)}\cap B_2&\ldots& C_{2,\sigma_2(k)}\cap B_2\cr
\vdots & \vdots & \vdots & \ddots & \vdots\cr
B_k & C_{k,\sigma_k(1)}\cap B_k & C_{k,\sigma_k(2)}\cap B_k & \ldots & C_{k,\sigma_k(k)}\cap B_k}=:B\cap C^{\sigma}.
\end{displaymath}	
	where $(B_1,\ldots,B_k)$ are the blocks of $B=B_{|[n]}(t-)$ listed in order of their least element, with $k-\#B$ empty sets at the end of the list. 
	\begin{itemize}
	\item if $B'\neq B$, then $B_{|[n]}(t)=B'$;
	\item if $B'=B$, $B_{|[n]}(t)=B_{|[n]}(t-)$.
\end{itemize}
\end{itemize}
\begin{prop}\label{prop:poissonian rates}The above process $B$ is a Markov process on $\mathcal{P}^{(k)}$ with transition matrix $Q^{\nu}$ defined by theorem \ref{thm:existence continuous-time}.\end{prop}
\begin{proof}
This is clear from the consistency of both the paintbox process $\varrho_{\nu}$ and the $Q^{\nu}_n$-matrices for every $n$ and the fact that, by this construction, for any $n$ such that $B_{|[n]}(t)=\pi$ then $B_{[n]|[m]}(t)=D_{m,n}(\pi)$ for all $m<n$ and $B_{|[p]}(t)\in D^{-1}_{n,p}(\pi)$ for all $p>n$.
\end{proof}
Let $\mathbb{P}_t$ be the semi-group of a $\varrho_{\nu}$-Markov process $B(\cdot)$, i.e.\ for any continuous $\varphi:\mathcal{P}^{(k)}\rightarrow\mathbb{R}$
$$\mathbb{P}_t\varphi(\pi):=\mathbb{E}_{\pi}\varphi(B(t)),$$
the expectation of $\varphi(B(t))$ given $B(0)=\pi.$
\begin{cor}\label{cor:feller}A $\varrho_{\nu}$-Markov process has the Feller property, i.e.
\begin{itemize}
	\item for each continuous function $\varphi:\mathcal{P}^{(k)}\rightarrow\mathbb{R}$, for each $\pi\in\mathcal{P}$ one has 
	$$\lim_{t\downarrow0}\mathbb{P}_t\varphi(\pi)=\varphi(\pi),$$
	\item for all $t>0$, $\pi\mapsto\mathbb{P}_{t}\varphi(\pi)$ is continuous.
\end{itemize}\end{cor}
\begin{proof}
The proof follows the same program as the proof of corollary 6 in \cite{Berestycki2004}.

Let $C_f:=\{f:\mathcal{P}^{(k)}\rightarrow\mathbb{R}:\exists n\in\mathbb{N}\mbox{ s.t. } \pi_{|[n]}=\pi'_{|[n]}\Rightarrow f(\pi)=f(\pi')\}$ be a set of functions which is dense in the space of continuous functions from $\mathcal{P}^{(k)}\rightarrow\mathbb{R}$.  It is clear that for $g\in C_f$, $\lim_{t\downarrow0}\mathbb{P}_t g(\pi)=g(\pi)$ since the first jump-time of $B(\cdot)$ is an exponential variable with finite mean.  The first point follows for all continuous functions $\mathcal{P}^{(k)}\rightarrow\mathbb{R}$ by denseness of $C_f$.  

For the second point, let $\pi,\pi'\in\mathcal{P}^{(k)}$ such that $d(\pi,\pi')<1/n$ and use the same Poisson point process $P$ to construct two $\varrho_{\nu}$-Markov processes, $B(\cdot)$ and $B'(\cdot)$, with starting points $\pi$ and $\pi'$ respectively.  By the construction, $B_{|[n]}=B'_{|[n]}$ and $d(B(t),B'(t))<1/n$ for all $t\geq0$.  It follows that for any continuous $g$, $\pi\mapsto\mathbb{P}_t g(\pi)$ is continuous.
\end{proof}
This allows us to characterize the $\varrho_{\nu}$-Markov process in terms of its infinitesimal generator.  Let $B:=(B(t),t\geq0)$ be the $\varrho_{\nu}$-Markov process on $\mathcal{P}^{(k)}$ with transition rates characterized by $(q_n)_{n\in\mathbb{N}}$ as in \eqref{eq:rates}.  The infinitesimal generator, $\mathcal{A}$, of $B$ is given by
$$\mathcal{A}(f)(\pi)=\int_{\mathcal{P}^{(k)}}f(\pi')-f(\pi)Q^{\nu}(\pi,d\pi'),$$ for every $f\in C_f$.
\section{Asymptotic frequencies}\label{section:asymptotic frequencies}
\begin{defn}\label{defn:asymptotic frequencies}A subset $A\subset\mathbb{N}$ is said to have asymptotic frequency $\lambda$ if
\begin{equation}\lambda:=\lim_{n\rightarrow\infty}\frac{\#\{i\leq n:i\in A\}}{n}\label{eq:asymptotic frequencies}\end{equation}
exists, and a random partition $B:=(B_1,B_2,\ldots)\in\mathcal{P}$ is said to have asymptotic frequencies if each block of $B$ has asymptotic frequency almost surely.\end{defn}
Adopting the notation of Berestycki \cite{Berestycki2004}, let $\Lambda(B)=(\lVert B_1\lVert,\lVert B_2\lVert,\ldots)^{\downarrow}$ be the decreasing arrangement of asymptotic frequencies of a partition $B=(B_1,B_2,\ldots)\in\mathcal{P}$ which possesses asymptotic frequencies, some of which could be 0.

According to Kingman's representation theorem (theorem 2.2, \cite{Pitman2005}) any exchangeable random partition of $\mathbb{N}$ possesses asymptotic frequencies.  Intuitively, this is a consequence of generating an exchangeable random partition of $\mathbb{N}$ by the paintbox process.  

The process described in section \ref{section:markov chain} only assigns positive probability to transitions involving two partitions with at most $k$ blocks.  From the Poissonian construction of the transition rates in section \ref{section:poissonian construction} it is evident that the states of $B=(B(t),t\geq0)$ will have at most $k$ blocks almost surely. Moreover, the description of the transition rates in terms of the paintbox process allows us to describe the associated measure-valued process of $B:=(B(t),t\geq0)$ characterized by $\lambda$ and $\nu$.
\subsection{Poissonian construction}\label{section:measure-valued poissonian}
Consider the following Poissonian construction of a measure-valued process $X:=(X(t),t\geq0)$ on $\masspartitionk$.  For any $k\in\mathbb{N}$, $\lambda>0$ and $\nu$ as above, let $P'=\{(t,P'_1,\ldots,P'_k)\}\subset\mathbb{R}^+\times\prod_{i=1}^k\masspartitionk$ be a Poisson point process with intensity measure $dt\otimes\lambda\nu^{(k)},$ where $\nu^{(k)}$ is the product measure $\nu\otimes\ldots\otimes\nu$ on $\prod_{i=1}^k\masspartitionk$.

Construct a process $X:=(X(t),t\geq0)$ on $\masspartitionk$ by generating $p_0$ from some probability distribution on $\masspartitionk$.  Put $X(0)=p_0$ and
\begin{itemize}
	\item if $t$ is not an atom time for $P'$, then $X(t)=X(t-)$;
	\item if $t$ is an atom time for $P'$ so that $(t,P'_1,\ldots,P'_k)\in P'$, with $P'_j=(P^j_1,\ldots,P^j_k)$ for each $j=1,\ldots,k$, and $X(t-)=(x_1,\ldots,x_k)\in\masspartitionk$, then, independently of $(X(s),s<t)$ and $(t,P'_1,\ldots,P'_k)$, generate $\sigma_1,\ldots,\sigma_k$ i.i.d.\ uniform random permutations of $[k]$ and construct $X(t)$ from the marginal column totals of 
	\begin{displaymath}
\bordermatrix{\text{}&P^{\subdot}_1 & P^{\subdot}_2 & \ldots & P^{\subdot}_k\cr
x_1 & x_1P^1_{\sigma_1(1)} & x_1P^1_{\sigma_1(2)}&\ldots& x_1P^1_{\sigma_1(k)}\cr
x_2 & x_2P^2_{\sigma_2(1)} & x_2P^2_{\sigma_2(2)}&\ldots& x_2P^2_{\sigma_2(k)}\cr
\vdots & \vdots & \vdots & \ddots & \vdots\cr
x_k & x_kP^k_{\sigma_k(1)} & x_kP^k_{\sigma_k(2)} & \ldots & x_kP^k_{\sigma_k(k)}}.
\end{displaymath}	
i.e.\ put $X(t)=(P^{\subdot}_1, P^{\subdot}_2,\ldots, P^{\subdot}_k)^{\downarrow}:=\left(\sum_{i=1}^k x_iP^i_{\sigma_i(j)},1\leq j\leq k\right)^{\downarrow}$.
\end{itemize}
\begin{thm} Let $X:=(X(t),t\geq0)$ be the process constructed above.  Then $X=_{\mathcal{L}}\Lambda(B)$ where $B:=(B(t),t\geq0)$ is the $\varrho_{\nu}$-Markov process from theorem \ref{thm:existence continuous-time}.\end{thm}
\begin{proof} Fix $k\in\mathbb{N}$ and let $\nu(\cdot)$ be a measure on $\masspartitionk$.

In the description of the sample paths of $B$ in section \ref{section:continuous-time process}, note that generating\\ $(C_1,\ldots,C_k)\sim\varrho_{\nu}^{(k)}$ is equivalent to first generating $s_i\sim\nu$ independently for each $i=1,\ldots,k$, then generating random partitions $C_i$ by sampling from $s_i$ for each $i=1,\ldots,k$.  Finally, $B'_i$ is set equal to the marginal total of column $i$ of the matrix $B\cap C^{\sigma}$, where $\sigma:=(\sigma_1,\ldots,\sigma_k)$ is an i.i.d.\ collection of uniform random permutations of $[k]$.  Hence, we can couple the two processes $X$ and $B$ together using the Poisson point process $P'$ described above.

Let $X$ evolve according to the Poisson point process $P'$ on $\mathbb{R}^+\times\prod_{i=1}^k\masspartitionk$ as described above.  Let $B$ evolve by the modification that if $t$ is an atom time of $P'$ then we obtain partitions $(C_1,\ldots, C_k)$ by sampling $X^i:=(X^i_1,X^i_2,\ldots)$ i.i.d.\ from $P'_i$ for each $i=1,\ldots,k$, i.e.\
$$\mathbb{P}(X^i_1=j|P'_i)=P^i_j,$$
 and defining the blocks of $C_i$ as the equivalence classes of $X^i$.  Constructed in this way, $\lVert C_{ij}\lVert= P^i_j$ almost surely for each $i,j=1,\ldots,k$ and $(C_1,\ldots,C_k)\sim\varrho_{\nu}^{(k)}$.  
 
After obtaining the $C_i$, generate, independently of $B,C_1,\ldots,C_k,P'$, i.i.d.\ uniform permutations $\sigma_1,\ldots,\sigma_k$ of $[k]$ and proceed as in the construction of section \ref{section:poissonian construction} where $B,C_1,\ldots, C_k$ are arranged in the matrix $B\cap C^{\sigma}$ and the blocks of $B'$ are obtained as the marginal column totals of $B\cap C^{\sigma}$.  The $(i,j)$th entry of $B\cap C^{\sigma}$ is $C_{i,\sigma_i(j)}\cap B_i$ for which we have $\lVert C_{i,\sigma_i(j)}\cap B_i\lVert=\lVert C_{i,\sigma_i(j)}\lVert\lVert B_i\lVert=x_iP^i_{\sigma_i(j)}$ a.s.

By this construction, $B(t)$ is constructed according to a Poisson point process with the same law as that described in section \ref{section:poissonian construction}, and $B(t)$ possesses ranked asymptotic frequencies which correspond to $X(t)$ almost surely for all $t\geq0$.
\end{proof}
\begin{cor}$X(t):=(\Lambda(B(t)),t\geq0)$ exists almost surely.\end{cor}
\subsection{Equilibrium measure}\label{section:ranked-mass equilibrium measure}
Just as the process $(B(t),t\geq0)$ on $\mathcal{P}^{(k)}$ converges to a stationary distribution, so does its associated measure-valued process $(X(t),t\geq0)$ from section \ref{section:measure-valued poissonian}.
\begin{thm} The associated measure-valued process $X$ for a $\varrho_{\nu}$-Markov process with unique stationary measure $\theta(\cdot;\nu)$ has equilibrium measure $\tilde{\theta}(\cdot;\nu),$ the distribution of the ranked frequencies of a $\theta(\cdot;\nu)$-partition.\end{thm}
\begin{proof}
Proposition 1.4 in \cite{BertoinPIMS} states that if a sequence of exchangeable random partitions converges in law on $\mathcal{P}$ to $\pi_{\infty}$ then its sequence of ranked asymptotic frequencies converges in law to $\vert\pi_{\infty}\vert^{\downarrow}$.  Hence, from corollary \ref{cor:continuous-time stationary} we have that $X$ has equilibrium distribution given by the ranked asymptotic frequencies of a $\theta(\cdot;\nu)$-partition.
\end{proof}
\section{The $(\alpha,k)$-Markov process}\label{section:(alpha,k)-process}
Pitman \cite{Pitman2005} discusses a two-parameter family of infinitely exchangeable random partitions called the $(\alpha,\theta)$ process which has finite-dimensional distributions
\begin{equation}p_n(B;\alpha,\theta):=\frac{(\theta/\alpha)^{\uparrow\#B}}{\theta^{\uparrow n}}\prod_{b\in B}-(-\alpha)^{\uparrow\#b},\label{eq:alpha-theta EPPF}\end{equation}
for $(\alpha,\theta)$ satisfying either
\begin{itemize}
	\item $\alpha=-\kappa<0$ and $\theta=m\kappa$ for some $m=1,2,\ldots,$ or
	\item $0\leq\alpha\leq1$ and $\theta>-\alpha$.
\end{itemize}
For $k\in\mathbb{N}$ and $\alpha>0$, a $(-\alpha,k\alpha)$ partition has finite-dimensional distributions
\begin{equation}\rho_n(B;\alpha,k)=\frac{k!}{(k-\#B)!}\frac{\prod_{b\in B}\Gamma(\alpha+\#b)/\Gamma(\alpha)}{\Gamma(k\alpha+n)/\Gamma(k\alpha)}\label{eq:DM fidi}\end{equation} whose support is $\partitionsnk$.

The distribution of the ranked asymptotic frequencies of an $(\alpha,\theta)$ partition is called the Poisson-Dirichlet distribution with parameter $(\alpha,\theta)$, written $\PD(\alpha,\theta)$.

For notational convenience, introduce the $\alpha$-permanent \cite{McCullagh2005} of an $n\times n$ matrix $K$, \begin{equation}\per_{\alpha}K=\sum_{\sigma\in\mathcal{S}_n}\alpha^{\#\sigma}\prod_{i=1}^n K_{i,\sigma(i)},\notag\end{equation}
where $\#\sigma$ is the number of cycles of the permutation $\sigma$, and note that when $B\in\mathcal{P}_{[n]}$ is regarded as a matrix,
\begin{equation}\per_{\alpha}B=\prod_{b\in B}\per_{\alpha}B_{|b}=\prod_{b\in B}\Gamma(\alpha+\#b)/\Gamma(\alpha),\label{eq:permanent 2}\end{equation}
which allows us to write \eqref{eq:DM fidi} as
\begin{equation}\rho_n(B;\alpha,k)=\frac{k!}{(k-\#B)!}\frac{\per_{\alpha}B}{(k\alpha)^{\uparrow n}},\end{equation}
where $(\beta)^{\uparrow n}=\beta(\beta+1)\cdots(\beta+n-1).$

We now consider a specific sub-family of reversible $\varrho_{\nu}$-Markov processes for which the transition probabilities can be written down explicitly.  For $k\in\mathbb{N}$ and $\alpha>0$, let $\nu$ be the $\PD(-\alpha/k,\alpha)$ distribution on $\masspartitionk$ and define transition probabilities according to the matrix construction based on $\nu$ as in section \ref{section:markov chain}.  We call this process the $(\alpha,k)$-Markov process.
\begin{prop}\label{prop:(alpha,k) tps} The $(\alpha,k)$-Markov process has finite-dimensional transition probabilities
\begin{eqnarray}
p_n(B,B';\alpha,k)&=&\frac{k!}{(k-\#B')!}\prod_{b\in B}\frac{\prod_{b'\in B'}\Gamma(\alpha/k+\#(b\cap b'))/\Gamma(\alpha/k)}{\Gamma(\alpha+\#b)/\Gamma(\alpha)}\label{eq:(alpha,k) tp1}\\
&=&\frac{k!}{(k-\#B')!}\frac{\per_{\alpha/k}(B\wedge B')}{\per_{\alpha}B}.\label{eq:(alpha,k) tp2}\end{eqnarray}\end{prop}
\begin{proof}
Theorem 3.2 and definition 3.3 from \cite{Pitman2005} shows that the distribution of $B\sim\varrho_{\nu}$ where $\nu=\PD(-\alpha/k,\alpha)$ is 
$$\rho_n(B;\alpha/k,k)=\frac{k!}{(k-\#B)!}\frac{\per_{\alpha/k}B}{(\alpha)^{\uparrow n}}.$$  Combining this and \eqref{eq:paintbox tps} yields \eqref{eq:(alpha,k) tp1}; \eqref{eq:(alpha,k) tp2} follows from \eqref{eq:permanent 2}.
\end{proof}
\begin{prop}For each $(\alpha,k)\in\mathbb{R}^+\times\mathbb{N}$ and $n\in\mathbb{N}$, $p_n(\cdot,\cdot;\alpha,k)$ defined in proposition \ref{prop:(alpha,k) tps} is reversible with respect to \eqref{eq:DM fidi} with parameter $(\alpha,k)$.\end{prop}
\begin{proof}
Let $\rho_n(\cdot;\alpha,k)$ be the distribution with parameter $(\alpha,k)$ defined in \eqref{eq:DM fidi}, and $p_n(\cdot,\cdot;\alpha,k)$ be as defined in \eqref{eq:(alpha,k) tp1}.  For any $B,B'\in\partitionsnk$, it is immediate that
\begin{equation}\rho_n(B;\alpha,k)p_n(B,B';\alpha,k)=\rho_n(B';\alpha,k)p_n(B',B;\alpha,k),\end{equation}
which establishes reversibility.
\end{proof}
Bertoin \cite{Bertoin2008} discusses some reversible EFC processes which have $\PD(\alpha,\theta)$ distribution as their equilibrium measure, for $0<\alpha<1$ and $\theta>-\alpha$.  Here we have shown reversibility with respect to $\PD(\alpha,\theta)$ for $\alpha<0$ and $\theta=-m\alpha$ for $m\in\mathbb{N}$.

The construction of the continuous-time process is a special case of the procedure in section \ref{section:continuous-time process}.  The measure-valued process $(X(t),t\geq0)$ based on the $(\alpha,k)$-Markov process has unique stationary measure $\PD(-\alpha,k\alpha),$ the distribution of the ranked frequencies of a partition with finite-dimensional distributions as in \eqref{eq:DM fidi} with parameter $(\alpha,k)$.
\section{Discussion}
The paths of the $\varrho_{\nu}$-Markov process discussed above are confined to $\mathcal{P}^{(k)}$.  Unlike the EFC-process \cite{Berestycki2004}, which has a natural interpretation as a model in certain physical sciences, the $\varrho_{\nu}$-Markov process has no clear interpretation as a physical model.  However, the matrix construction introduced in section \ref{section:markov chain} leads to transition rates which admit a closed form expression in the case of the $(\alpha,k)$-Markov process.  

The $(\alpha,k)$ class of models could be useful as a statistical model for relationships among statistical units which are known to fall into one of $k$ classes.  In statistical work, it is important that any observation has positive probability under the specified model.  The $(\alpha,k)$-process assigns positive probability to all possible transitions and so any observed sequence of partitions in $\partitionsnk$ will have positive probability for any choice of $\alpha>0$.  In addition, the model is exchangeable, consistent and reversible, particularly attractive mathematical properties which could have a natural interpretation in certain applications.  Future work is intended to explore applications for this model, as well as develop some of the tools necessary for its use in statistical inference.

\acks
This work is supported by the National Science Foundation grant no. DMS-0906592.  I also thank Peter McCullagh for his many helpful comments throughout the editing of this paper.

\end{document}